\documentclass[leqno 12pts]{amsart}
\usepackage{amsmath, amsthm, amscd, amsfonts, amssymb, graphicx, color}
\usepackage[bookmarksnumbered, colorlinks, plainpages]{hyperref}
\usepackage{tikz}
\usepackage{enumerate}

\newtheorem{theorem}{Theorem}[section]
\newtheorem{lemma}[theorem]{Lemma}
\newtheorem{proposition}[theorem]{Proposition}

\newtheorem{corollary}[theorem]{Corollary}
\newtheorem{definition}[theorem]{Definition}

\usepackage{comment}
\excludecomment{mysection}

\begin{document}
	
\title{Absolute compatibility and poincar\'{e} sphere}
\author{Anil Kumar Karn}
	
\address{School of Mathematical Sciences, National Institute of Science Education and Research, HBNI, Bhubaneswar, P.O. - Jatni, District - Khurda, Odisha - 752050, India.}

\email{\textcolor[rgb]{0.00,0.00,0.84}{anilkarn@niser.ac.in}}
	
	
\subjclass[2010]{Primary 46L05; Secondary 46B40, 46B42.}
	
\keywords{von Neumann algebra, absolute compatibility, strict projection, Poincar\'{e} sphere}
	
\begin{abstract}
	In this paper, we introduce the notion of strict projections and prove that an absolutely compatible pair of strict elements in a von Neumann algebra $\mathcal{M}$ unitarily equivalent to the elements $ \left((p - x_0) \otimes I_2 \right) P_0 + (x_0 \otimes I_2) P$, $\left((p - x_0) \otimes I_2 \right) P_0 + (x_0 \otimes I_2) P'$ of $M_2(\mathcal{M}_0)$ where $\mathcal{M}_0$ is an abelian von Neumann algebra, $x_0$ is a strict element of $\mathcal{M}_0^+$, $P_0 = \begin{bmatrix} 0 & 0 \\ 0 & 1 \end{bmatrix} \in M_2(\mathcal{M}_0)$ and $P$ is a strict projection in $M_2(\mathcal{M}_0)$. We also discuss the geometric form of this representation when $\mathcal{M} = \mathbb{M}_2$.
\end{abstract}

\maketitle 

\section{Introduction} 
Let $A$ be a C$^*$-algebra. A pair of elements $a, b \in A$ is said to be orthogonal, if $a^* b = 0 = a b^*.$ Orthogonal pair of positive elements play an important role in the theory of C$^*$-algebras. For example, it follows from the functional calculus that every self-adjoint element $a \in A_{sa}$ has a unique decomposition: $a = a^+ - a^-$ in $A^+$, where $a^+$ is algebraically orthogonal to $a^-$. Recently, the author obtained an order theoretic characterization of algebraic orthogonality among positive elements of a C$^*$-algebra \cite{K16, K18}. 

Orthogonal pairs of positive elements of norm less than or equal to $1$ exhibit an interesting property. Let $A$ be a unital C$^*$-algebra with unity $1_A$. For a pair of elements $0 \le a, b \le 1_A$ in $A$, we have $a b = 0$ ($a$ is orthogonal to $b$) if and only if $a + b \le 1_A$ and $\vert a - b \vert + \vert 1_A - a - b \vert = 1_A$ \cite[Proposition 4.1]{K18}. We isolate the later part and propose the following definition: A pair of elements $0 \le a, b \le 1_A$ in $A$ is said to be \emph{absolutely compatible}, (\cite{K18}), if $$\vert a - b \vert + \vert 1_A - a - b \vert = 1_A.$$ 
The notion of absolute compatibility was introduced as an instrument to produce a spectral decomposition theorem in the context of `absolute order unit spaces' \cite{K18}. It is further interesting to note that a projection $p$ in a C$^*$-algebra $A$ is absolutely compatible with a positive element $a$ of $A$ with $\Vert a \Vert \le 1$ if, and only if, $a p = p a$ \cite[Proposition 4.8]{K18}. Therefore, the notion of absolute compatibility appears to be interesting on its own. Keeping this point of view, the author, along with Jana and Peralta, initiated a study of absolute compatibility in operator algebras \cite{JKP, JKP1}. 

Let $\mathcal{M}$ be a von Neumann algebra and let $0 \le a \le 1_{\mathcal{M}}$. We write 
$$s(a) := \sup \lbrace p \in \mathcal{P}(\mathcal{M}): p \le a \rbrace$$ 
and
$$n(a) := \sup \lbrace p \in \mathcal{P}(\mathcal{M}): p a = 0 \rbrace$$
where $\mathcal{P}(\mathcal{M})$ denotes the set of all projections in $\mathcal{M}$. We also recall that the \emph{range projection} $r(a)$ of $a \in \mathcal{M}^+$ is defined as 
$$r(a) := \inf \lbrace p \in \mathcal{P}(\mathcal{M}): a \le \Vert a \Vert p \rbrace = \inf \lbrace p \in \mathcal{P}(\mathcal{M}): a p = a \rbrace.$$
For $0 \le a \le 1_{\mathcal{M}}$, we say that $a$ is \emph{strict} in $\mathcal{M}$, if $s(a) = 0$ and $n(a) = 0$ (that is, $r(a) = 1_{\mathcal{M}}$). Note that if $a$ is strict, then so is $1 - a$. This notion was studied in \cite{JKP1, K20}. 

In this paper, we extend this notion to a general element $a \in \mathcal{M}$. We say that $a$ is strict, if $\vert a \vert$ is strict in the above said sense. Note that in this case, we can find a unitary $u \in \mathcal{M}$ such that $a = u \vert a \vert$.

In \cite{JKP}, it was proved that an absolutely compatible pair of positive elements in a von Neumann algebra has a (matricial) decomposition as a direct sum of commuting and `strict' components. Let $0 \le a, b \le 1_{\mathcal{M}}$ such that $a$ is absolutely compatible with $b$. Then there exist mutually orthogonal projections $p_1, p_2, s, n_1, n_2 \in \mathcal{P}(\mathcal{M})$ with $p_1 + p_2 + s + n_1 + n_2 = 1_{\mathcal{M}}$ such that  
$$a = p_1 \oplus a_1 \oplus a_2 \oplus 0 \oplus a_3$$
and 
$$b = b_1 \oplus p_2 \oplus b_2 \oplus b_3 \oplus 0$$
with respect to $\{ p_1, p_2, s, n_1, n_2 \}$ where $a_2$ and $b_2$ are strict and absolutely compatible in $s\mathcal{M}s$ \cite[Theorem 2.10]{JKP}. Thus the study of absolutely compatible pair of elements reduces to such strict pairs. 

Let $p$ and $q$ be any two projections in a von Neumann algebra $\mathcal{M}$ and let $\lambda$ be a real number with $0 \le \lambda \le 1$. It is easy to show that $a := (1 - \lambda) p + \lambda q$ is absolutely compatible with $b := (1 - \lambda) p + \lambda q'$ where $q' = 1 - q$. In this paper we show that any absolutely compatible pair of strict elements in $\mathcal{M}$ takes this form for a suitable pair of projections $p$ and $q$. However, we note that if $p \wedge q$ is non-zero, then $a$ is not strict. Similarly, $\lambda = 0$ or $\lambda = 1$ reduces these elements to projections. We find a specific form for $p$ and $q$ to prove strictness of $a$ and $b$.

In \cite{K20}, we proved the following description of an absolutely compatible pair of strict elements in a von Neumann algebra. 
\begin{theorem}\label{1}
	Let $\mathcal{M}$ be a von Neumann algebra and assume that $a, b \in [0, 1]_{\mathcal{M}}$ be strict and commuting pair such that $a^2 + b^2 \le 1$ with $a^2 + b^2$ strict. Put $a_1 = \begin{bmatrix} a^2 & a b \\ a b & 1 - a^2 \end{bmatrix}$ and $b_1 = \begin{bmatrix} b^2 & - a b \\ - a b & 1 - b^2 \end{bmatrix}$. Then $a_1, b_1 \in [0, 1]_{M_2(\mathcal{M})}$ and $a_1$ is absolutely compatible with $b_1$.
\end{theorem} 
Conversely, we have 
\begin{theorem}\label{2}
	Let $\mathcal{M}$ be a von Neumann algebra with the underlying Hilbert space $H$. Assume that $a, b \in [0, 1]_{\mathcal{M}}$ be a strict and absolutely compatible pair. Put $p = 1 - r(a \circ b)$. 
	\begin{enumerate}
		\item Then $p H$ is isometrically isomorphic to $(1 - p) H$. In particular, $H \equiv K \oplus K$, where $K = p H$. 
		\item There exist strict elements $a_1, b_1 \in [0, p] \cap \mathcal{M}$ with $a_1 b_1 = b_1 a_1$, $a_1 + b_1 \le p$ together with $a_1 + b_1$ strict in $p \mathcal{M} p$; and a unitary $U: H \to K \oplus K$ such that 
		$$a = U^* \begin{bmatrix} a_1 & (a_1 b_1)^{\frac{1}{2}} \\ (a_1 b_1)^{\frac{1}{2}} & p - a_1 \end{bmatrix} U \ \textrm{and} \ b = U^* \begin{bmatrix} b_1 & - (a_1 b_1)^{\frac{1}{2}} \\ - (a_1 b_1)^{\frac{1}{2}} & p - b_1 \end{bmatrix} U.$$
	\end{enumerate}
\end{theorem}
In this paper, we apply Theorem \ref{2} to prove that if $a$ and $b$ are strict and absolutely compatible in a von Neumann algebra $\mathcal{M}$ and if $\mathcal{M}_1$ is the von Neumann subalgebra generated by $a$ and $b$ then there exists an abelian von Neumann subalgebra $\mathcal{M}_0$ of $\mathcal{M}$ such that $\mathcal{M}_1$ is unitally isomorphic to $M_2(\mathcal{M}_0)$ (Theorem \ref{8}). This result reduces our effort to work on absolutely compatible pair of strict elements in $M_2(\mathcal{M}_0)$ where $\mathcal{M}_0$ is an abelian von Neumann algebra. 

We introduce the notions of strict unitaries and strict projections in $M_2(\mathcal{M}_0)$ where $\mathcal{M}_0$ is an abelian von Neumann algebra and describe an absolutely compatible pair of strict element in a von Neumann algebra as a operator convex combinations of strict projections (up to suitable unitaries). More precisely, the main results of the paper are the following:
\begin{theorem}\label{9}
	Let $\mathcal{M}_0$ be an abelian von Neumann algebra and assume that $A, B \in M_2(\mathcal{M}_0)$ with $0 \le A, B \le I_2$ be such that 
	$$A = \left((1 - x_0) \otimes I_2 \right) P_0 + (x_0 \otimes I_2) P$$
	and 
	$$B = \left((1 - x_0) \otimes I_2 \right) P_0 + (x_0 \otimes I_2) P'$$ 
	for some strict element $x_0 \in \mathcal{M}_0^+$ and a strict projection $P$ in $M_2(\mathcal{M}_0)$. (Here $I_2 = \begin{bmatrix} 1 & 0 \\ 0 & 1 \end{bmatrix}$ is the unity of $M_2(\mathcal{M}_0)$, $P_0 = \begin{bmatrix} 0 & 0 \\ 0 & 1 \end{bmatrix}$ is a pivotal projection in $M_2(\mathcal{M}_0)$ and $P' = I_2 - P$.) Then $A$ and $B$ are strict elements of $M_2(\mathcal{M}_0)^+$ and $A$ is absolutely compatible with $B$.
\end{theorem}
\begin{theorem}\label{10}
	Let $\mathcal{M}$ be a von Neumann algebra with the underlying Hilbert space $H$. Assume that $a, b \in [0, 1]_{\mathcal{M}}$ be a strict and absolutely compatible pair and let $\mathcal{M}_1$ be the von Neumann subalgebra of $\mathcal{M}$ generated by $a$ and $b$. 
	\begin{enumerate}
		\item There exists a halving projection $p \in \mathcal{M}$ and a unitary $U_0: K \oplus K \to H$ such that $U_0^* \mathcal{M}_1 U_0$ is unitally $*$-isomorphic to $M_2(\mathcal{M}_0)$ for some abelian von Neumann algebra $\mathcal{M}_0$ acting on $K$ where $K = p H$. 
		\item There exists a strict projection $P$ in $M_2(\mathcal{M}_0)$ and a strict element $x_0$ in $\mathcal{M}_0^+$ such that 
		$$U_0^* a U_0 = \left((p - x_0) \otimes I_2 \right) P_0 + (x_0 \otimes I_2) P$$
		and 
		$$U_0^* b U_0 = \left((p - x_0) \otimes I_2 \right) P_0 + (x_0 \otimes I_2) P'$$
		where $I_2 = \begin{bmatrix} p & 0 \\ 0 & p \end{bmatrix}$ is the unity of $M_2(\mathcal{M}_0)$, $P_0 = \begin{bmatrix} 0 & 0 \\ 0 & p \end{bmatrix}$ is a projection in $M_2(\mathcal{M}_0)$ and $P' = I_2 - P$.
	\end{enumerate}
\end{theorem}
We apply these results to $\mathbb{M}_2$ and obtain a geometric form of absolutely compatible pair of strict elements in $\mathbb{M}_2$ represented as elements in the interior of the Poincar\'{e} sphere. Recall that the Poincar\'{e} sphere can be considered as the set of rank one projections in $\mathbb{M}_2$. We deduce that absolutely compatible pair of strict elements in $\mathbb{M}_2$ replicates the orthogonality of (rank one) projections when described on an sphere touching the Poincar\'{e} sphere from inside. We understand that this property may put more light on the theory of polarization in optics as well as some other fields of physics. We propose a model for physical verification of this idea. 

\section{ Absolute compatibility and strict projections}

\begin{theorem}\label{8}
	Let $\mathcal{M}$ be a von Neumann algebra with the underlying Hilbert space $H$. Assume that $a, b \in [0, 1]_{\mathcal{M}}$ be a strict and absolutely compatible pair and let $\mathcal{M}_1$ be the von Neumann subalgebra of $\mathcal{M}$ generated by $a$ and $b$. Then $\mathcal{M}_1$ is unitally isomorphic to $M_2(\mathcal{M}_0)$ for a suitable abelian von Neumann algebra $\mathcal{M}_0$.
\end{theorem}
\begin{proof}
	Put $p = 1 - r(a \circ b)$ and $K = p H$. It follows from the proof of \cite[Theorem 1.4]{K20} that there exists a partial isometry $u \in \mathcal{M}$ such that $u^* u = p$ and $u u^* = 1 - p$. (That is, $p$ is a \emph{halving projection}.) Thus $u: p H \to (1 - p) H$ is a unitary operator and consequently, $W := \begin{bmatrix} 0 & u \\ p & 0 \end{bmatrix}: K \oplus K \to H$ is also a unitary operator. It further follows that there exist $a_0, b_0 \in [0, p] \cap  \mathcal{M}$ with $a_0 b_0 = b_0 a_0$, $a_0^2 + b_0^2 \le p$ together with $a_0$, $b_0$ and $a_0^2 + b_0^2$ strict in $p \mathcal{M}^+ p$ such that 
	$$W^* a W = \begin{bmatrix} a_0^2 & a_0 b_0 \\ a_0 b_0 & p - a_0^2 \end{bmatrix}; \quad W^* b W = \begin{bmatrix} b_0^2 & - a_0 b_0 \\ - a_0 b_0 & p - b_0^2 \end{bmatrix}.$$ 
	
	Let $\mathcal{M}_1$ be the von Neumann algebra generated by $a_1 := W^* a W$ and $b_1 := W^* b W$ and let $\mathcal{M}_0$ be the von Neumann algebra generated by $a_0$ and $b_0$. Then $\mathcal{M}_1 \subset M_2(\mathcal{M}_0)$. Also, as $a_0 b_0 = b_0 a_0$, we note that $\mathcal{M}_0$ is abelian. We show that $M_2(\mathcal{M}_0) \subset \mathcal{M}_1$. 
	
	We have $a_1 \circ b_1 = 0 \oplus (p - a_0^2 - b_0^2)$ and $r(p - a_0^2 - b_0^2) = p$ so that $0 \oplus p = r(a_1 \circ b_1) \in \mathcal{M}_1$. As $p \oplus p$ is the unity of $M_2(\mathcal{M}_0)$, we also have $p \oplus 0 \in \mathcal{M}_1$. Therefore, $a_0^2 \oplus 0 = (p \oplus 0) a_1 (p \oplus 0) \in \mathcal{M}_1$ and consequently, $a_0 \oplus 0 \in \mathcal{M}_1$. Similarly, we can show that $b_0 \oplus 0 \in \mathcal{M}_1$. Now it follows that $\begin{bmatrix} 0 & a_0 b_0 \\ 0 & 0 \end{bmatrix}, \begin{bmatrix} 0 & 0 \\ a_0 b_0 & 0 \end{bmatrix} \in \mathcal{M}_1$. 
	
	Next, we note that $\begin{bmatrix} \frac 1n  + a_0 & 0 \\ 0 & \frac 1n \end{bmatrix}$ is invertible in $\mathcal{M}_1$ with 
	$$\begin{bmatrix} \frac 1n  + a_0 & 0 \\ 0 & \frac 1n \end{bmatrix}^{-1} = \begin{bmatrix} (\frac 1n  + a_0)^{-1} & 0 \\ 0 & n \end{bmatrix}.$$ 
	Thus 
	\[ \begin{bmatrix} 0 & (\frac 1n  + a_0)^{-1} a_0 b_0 \\ 0 & 0 \end{bmatrix} = \begin{bmatrix} (\frac 1n  + a_0)^{-1} & 0 \\ 0 & n \end{bmatrix} \begin{bmatrix} 0 & a_0 b_0 \\ 0 & 0 \end{bmatrix} \in \mathcal{M}_1. \]
	As $(\frac 1n + a_0)^{-1} a_0$ converges $r(a_0) = p$ in the strong operator topology, we have $(\frac 1n + a_0)^{-1} a_0 b_0$ converges to $p b_0 = b_0$ in the strong operator topology. Thus 
	\[ \begin{bmatrix} 0 & b_0 \\ 0 & 0 \end{bmatrix} = SOT\lim_{n \to \infty} \begin{bmatrix} 0 & (\frac 1n  + a_0)^{-1} a_0 b_0 \\ 0 & 0 \end{bmatrix} \in \mathcal{M}_1. \] 
	Since $a_0 b_0 = b_0 a_0$, in the same way we can show that $\begin{bmatrix} 0 & a_0 \\ 0 & n \end{bmatrix} \in \mathcal{M}_1$. In a similar way, we can also get $\begin{bmatrix} 0 & 0 \\ a_0 & 0 \end{bmatrix}, \begin{bmatrix} 0 & 0 \\ b_0 & 0 \end{bmatrix} \in \mathcal{M}_1$. Further, 
	\[ \begin{bmatrix} 0 & p \\ 0 & 0 \end{bmatrix} = SOT\lim_{n \to \infty} \begin{bmatrix} 0 & (\frac 1n  + a_0)^{-1} a_0 \\ 0 & 0 \end{bmatrix} \in \mathcal{M}_1. \] 
	Similarly, $\begin{bmatrix} 0 & 0 \\ p & 0 \end{bmatrix} \in \mathcal{M}_1$. 
	
	Finally, let $x = \begin{bmatrix} x_{11} & x_{12} \\ x_{21} & x_{22} \end{bmatrix} \in M_2(\mathcal{M}_0)$ so that $x_{ij} \in \mathcal{M}_0, 1 \le i, j \le 2$. Then there exist nets $\lbrace q_{\lambda_{ij}} \rbrace$ of polynomials in two variables such that $x_{ij} = SOT-\lim q_{\lambda_{ij}} (a_0, b_0) \in \mathcal{M}_0$ for $1 \le i, j \le 2$. Thus 
	$$x = \sum x_{ij} \otimes e_{ij} = SOT\lim \sum q_{\lambda_{ij}} (a_0, b_0) \otimes e_{ij} \in \mathcal{M}_1.$$
	Hence $M_2(\mathcal{M}_0) \subset \mathcal{M}_1$.
\end{proof}
In this section, we introduce the notions of strict projections and strict unitaries in $M_2(\mathcal{M}_0)$ where $\mathcal{M}_0$ is an abelian von Neumann algebra and relate these notions to of absolute compatibility therein.
\begin{definition}
	Let $\mathcal{M}_0$ be an abelian von Neumann algebra and consider the von Neumann algebra $M_2(\mathcal{M}_0)$ of $2 \times 2$ matrices over $\mathcal{M}$. A unitary $U \in M_2(\mathcal{M}_0)$ is said to be a \emph{strict unitary}, if $U = \begin{bmatrix} u_1 & u_2 \\ u_3 & u_4 \end{bmatrix}$ and $u_1, u_2, u_3, u_4$ are strict in $\mathcal{M}_0$. Next, a projection  $P \in M_2(\mathcal{M}_0)$ is said to be a \emph{strict projection}, if $P = \begin{bmatrix} p_1 & p \\ p^* & p_2 \end{bmatrix}$ and $p_1$ is strict in $\mathcal{M}_0^+$ with $p_1 + p_2 = 1$.
\end{definition}
Note that if $P = \begin{bmatrix} p_1 & p \\ p^* & p_2 \end{bmatrix}$ is a strict projection, then $p^* p = p_1 p_2$. Thus $p$ is also strict in $\mathcal{M}_0$. Also, by the definition, $p_2$ is strict in $\mathcal{M}_0^+$ as well. 
\begin{proposition}\label{3}
	Let $\mathcal{M}_0$ be an abelian von Neumann algebra. Then $P$ is a strict projection in $M_2(\mathcal{M}_0)$ if and only if there exists a strict unitary $U \in M_2(\mathcal{M}_0)$ with $U = \begin{bmatrix} u_1 & u_2 \\ u_3 & u_4 \end{bmatrix}$ such that $P = \begin{bmatrix} u_1^* u_1 & u_1^*u_2 \\ u_2^*u_1 & u_2^*u_2 \end{bmatrix}$ and $P' := I_2 - P = \begin{bmatrix} u_3^* u_3 & u_3^*u_4 \\ u_4^*u_3 & u_4^*u_4 \end{bmatrix}$. Here $I_2$ is the identity of $M_2(\mathcal{M}_0)$.
\end{proposition}
\begin{proof}
	First assume that $P = \begin{bmatrix} p_1 & p \\ p^* & p_2 \end{bmatrix}$ is a strict projection. Then $\vert p \vert = p_1^{\frac 12} p_2^{\frac 12}$. Let $u$ be the unitary in $\mathcal{M}_0$ such that $p = u \vert p \vert$. Put $u_1 = u^* p_1^{\frac 12}$, $u_2 = p_2^{\frac 12}$, $u_3 = - u^* p_2^{\frac 12}$ and $u_4 = p_1^{\frac 12}$. Then $U = \begin{bmatrix} u_1 & u_2 \\ u_3 & u_4 \end{bmatrix}$ is a strict unitary in $M_2(\mathcal{M}_0)$ such that $P = \begin{bmatrix} u_1^* u_1 & u_1^*u_2 \\ u_2^*u_1 & u_2^*u_2 \end{bmatrix}$  and $P' = \begin{bmatrix} u_3^* u_3 & u_3^*u_4 \\ u_4^*u_3 & u_4^*u_4 \end{bmatrix}$. 
	
	Conversely if $U = \begin{bmatrix} u_1 & u_2 \\ u_3 & u_4 \end{bmatrix}$ is a strict unitary in $M_2(\mathcal{M}_0)$, then as $U^* U = I_2 = U U^*$ using matrix multiplications, we can verify that $P = \begin{bmatrix} u_1^* u_1 & u_1^*u_2 \\ u_2^*u_1 & u_2^*u_2 \end{bmatrix}$ is a strict projection and that $P' = \begin{bmatrix} u_3^* u_3 & u_3^*u_4 \\ u_4^*u_3 & u_4^*u_4 \end{bmatrix}$. 
\end{proof} 
\begin{lemma}\label{4}
	Let $\mathcal{M}_0$ be an abelian von Neumann algebra. Then $U = \begin{bmatrix} u_1 & u_2 \\ u_3 & u_4 \end{bmatrix}$ is a strict unitary in $M_2(\mathcal{M}_0)$ if and only if there exists a strict element $a_0 \in \mathcal{M}_0^+$ and unitaries $w_1, w_2, w_3 \in \mathcal{M}_0$ such that $u_1 = w_1 a_0$, $u_2 = w_2 (1 - a_0^2)^{\frac 12}$, $u_3 = w_3 (1 - a_0^2)^{\frac 12}$ and $u_4 = - w_1^* w_2 w_3 a_0$. 
\end{lemma} 
\begin{proof}
	As $U$ is unitary, we have $U^* U = I_2 = U U^*$. Thus as $\mathcal{M}_0$ be abelian, matricial computations yield 
	$$u_1^* u_1 + u_3^* u_3 = 1 = u_2^* u_2 + u_4^* u_4;$$
	$$u_1^* u_2 + u_3^* u_4 = 0;$$ 
	$$u_1^* u_1 + u_2^* u_2 = 1 = u_3^* u_3 + u_4^* u_4;$$
	$$u_1^* u_3 + u_2^* u_4 = 0.$$ 
	In other words, $u_1^* u_1 = u_4^* u_4$, $u_2^* u_2 = u_3^* u_3$, $u_1^* u_2 = - u_3^* u_4$ and $u_1^* u_3 = - u_2^* u_4$ with $u_1^* u_1 + u_2^* u_2 = 1$. Put $\vert u_1 \vert = a_0$. Then as $u_1$ is strict in $\mathcal{M}_0$, we have $a_0$ is strict in $\mathcal{M}_0^+$. Also, $\vert u_4 \vert = \vert u_1 \vert = a_0$ and $\vert u_2 \vert = \vert u_3 \vert = (1 - a_0^2)^{\frac 12}$. Consider the polar decompositions $u_i = w_i \vert u_i \vert$, $1 \le i \le 4$. Then $w_1$, $w_2$, $w_3$ and $w_4$ are unitaries in $\mathcal{M}_0$. Also, $u_1^* u_2 = - u_3^* u_4$ reduces to $a_0 (1 - a_0^2)^{\frac 12} (w_1^* w_2 + w_3^* w_4) = 0$. Thus as $a_0$ is strict, by \cite[Proposition 2.3(iv)]{JKP}, we conclude that $w_1^* w_2 + w_3^* w_4 = 0$. In other words, $w_4 = - w_1^* w_2 w_3$. Now, the verification of the converse part is routine.
\end{proof}
The next result is now follows immediately.
\begin{corollary}\label{5}
	Let $\mathcal{M}_0$ be an abelian von Neumann algebra. Then $P = \begin{bmatrix} p_1 & p \\ p^* & p_2 \end{bmatrix}$ is a strict projection in $M_2(\mathcal{M}_0)$ if and only if there exists a strict element $a_0 \in \mathcal{M}_0^+$ and a unitary $w \in \mathcal{M}_0$ such that $p_1 = a_0^2$, $p_2 = 1 - a_0^2$ and $p = w a_0 (1 - a_0^2)^{\frac 12}$. 
\end{corollary}
\begin{lemma}\label{6}
	Let $\mathcal{M}_0$ be an abelian von Neumann algebra. Then $P$ is a strict projection in $M_2(\mathcal{M}_0)$ if and only if there exists a strict unitary $U \in M_2(\mathcal{M}_0)$ such that $P = U^* \begin{bmatrix} 0 & 0 \\ 0 & 1 \end{bmatrix} U$. 
\end{lemma}
\begin{proof}
	By Corollary \ref{5}, $P = \begin{bmatrix} a_0^2 & w a_0 (1 - a_0^2)^{\frac 12} \\ w^* a_0 (1 - a_0^2)^{\frac 12} & 1 - a_0^2 \end{bmatrix}$ for some strict element $a_0 \in \mathcal{M}_0^+$ and a unitary $w \in \mathcal{M}_0$. Consider $U = \begin{bmatrix} (1 - a_0^2)^{\frac 12} & - w a_0 \\ a_0 & w (1 - a_0^2)^{\frac 12} \end{bmatrix}$. Then $U$ is a strict unitary in $M_2(\mathcal{M}_0)$ and we have $P = U^* \begin{bmatrix} 0 & 0 \\ 0 & 1 \end{bmatrix} U$. 
	
	Conversely, assume that $U = \begin{bmatrix} u_1 & u_2 \\ u_3 & u_4 \end{bmatrix}$ is a strict unitary in $M_2(\mathcal{M})$. Then $U^* \begin{bmatrix} 0 & 0 \\ 0 & 1 \end{bmatrix} U = \begin{bmatrix} u_3^* u_3 & u_3^* u_4 \\ u_4^* u_3 & u_4^* u_4 \end{bmatrix} = P$ (say). As $U$ is a strict unitary, we have $u_3$ and $u_4$ are strict elements in $\mathcal{M}_0$. Thus $P$ is a strict projection.
\end{proof}

\section{The main results}

Now we describe absolutely compatible pairs of strict elements in a von Neumann algebra in terms of strict projections. 
\begin{proof}[Proof of Theorem \ref{9}]
	First, we note that as $\mathcal{M}_0$ abelian, the centre of $M_2(\mathcal{M}_0)$ is given by 
	$$Z(M_2(\mathcal{M}_0)) = \lbrace x \otimes I_2: x \in \mathcal{M}_0 \rbrace.$$ 
	Now as $A - B = (x_0 \otimes I_2) (P - P')$ and as $P$ is a projection, we deduce that $\vert A - B \vert = x_0 \otimes I_2$. Again $I_2 - A - B = \left( (1 - x_0) \otimes I_2 \right) (I_2 - 2 P_0)$ and $P_0$ is also a projection, we further deduce that $\vert I_2 - A -B \vert = (1 - x_0) \otimes I_2$. Thus $\vert A - B \vert + \vert I_2 - A - B \vert = I_2$, that is, $A$ is absolutely compatible. 
	
	Next, we show that $A$ and $B$ are strict elements of $M_2(\mathcal{M}_0)^+$. Let $P = \begin{bmatrix} p_1 & p \\ p^* & p_2 \end{bmatrix}$. Since $P$ is a strict projection, we have $p_1$ is a strict element in $\mathcal{M}_0^+$, such that $p_1 + p_2 = 1$ and $p^* p = p_1 p_2$. Let $Q$ be a projection in $M_2(\mathcal{M}_0)$ such that $Q \le A$. Let $Q = \begin{bmatrix} q_1 & q \\ q^* & q_2 \end{bmatrix}$. As $\mathcal{M}_0$ is abelian and $Q$ is a positive element of $M_2(\mathcal{M}_0)$, we have $q^* q \le q_1 q_2$. Since $Q \le A$, we get, $q_1 \le x_0 p_1$ and $q_2 \le (1 - x_0) + x_0 p_2= 1 - x_0 p_1$. Thus $q_1 + q_2 \le 1$. 
	
	Since $Q$ is a projection, $Q^2 = Q$. Thus comparing entries, we get, 
	$$q_1 - q_1^2 = q^* q = q_2 - q_2^2; \qquad q (q_1 + q_2) = q.$$
	Thus 
	$$2 q_1 q_2 \ge 2 q^* q = q_1 + q_2 - q_1^2 - q_2^2$$
	so that $q_1 + q_2 \le (q_1 + q_2)^2$. As $q_1 + q_2 \le 1$, we also have $q_1 + q_2 \ge (q_1 + q_2)^2$ so that $q_0 := q_1 + q_2$ is a projection in $\mathcal{M}_0$. Put $Q_0 = q_0 \otimes I_2$. Then $Q_0$ is a central projection in $M_2(\mathcal{M}_0)$. Also a direct computation yields that $Q = Q^2 = Q Q_0$. Thus $Q \le Q_0 A$ and it follows that $q_1 \le x_0 q_0 p_1$ and 
	$$q_2 \le (1 - x_0) q_0 + x_0 q_0 p_2= q_0 - x_0 q_0 p_1.$$ 
	Since $q_0 = q_1 + q_2$, we get $q_1 = x_0 q_0 p_1$. Thus as $Q_0 A - Q \ge 0$, we must have $q = x_0 q_0 p$. Also as $Q_0 Q = Q$, we have $q_0 q_1 = q_1$, $q_0 q_2 = q_2$ and $q_0 q = q$. Thus 
	$$q^* q = q_1 - q_1^2 = q_1(q_0 - q_1) = q_1 q_2.$$ 
	Now, it follows that 
	$$x_0 q_0 p_1 q_2 = q_1 q_2 = q^* q = x_0^2 q_0 p^* p = x_0^2 q_0 p_1 p_2$$ 
	so that $x_0 q_0 p_1 (q_2 - x_0 p_2) = 0$. Since $x_0$ and $p_1$ are strict elements in $\mathcal{M}_0^+$, we conclude that $q_2 = x_0 q_0 p_2$. Then 
	$$q_0 = q_1 + q_2 = x_0 q_0 (p_1 + p_2) = x_0 q_0.$$ 
	Again invoking strictness of $x_0$ in $\mathcal{M}_0^+$, we deduce that $q_0 = 0$. Thus $Q = Q_0 Q = 0$. Hence $s(A) = 0$.
	
	Next, consider a projection $R$ in $M_2(\mathcal{M}_0)$ be such that $A \le R$. Then $R' \le  \left((p - x_0) \otimes I_2 \right) P_0' + (x_0 \otimes I_2) P'$. Thus, as above, we can show that $R' = 0$ so that $R = I_2$. Now it follows that $A$ is a strict element in $M_2(\mathcal{M}_0)^+$. In the same way, we can also show that $A$ is also a strict element in $M_2(\mathcal{M}_0)^+$.	
\end{proof}

\begin{proof}[Proof of Theorem \ref{10}]  
	(1) follows from \cite[Theorem 1.4]{K20} and Lemma \ref{8}. To prove (2), we again invoke \cite[Theorem 1.4]{K20} to deduce that there exist $a_0$ and $b_0$ in $\mathcal{M}_0^+$ such that $a_0^2 + b_0^2 \le p$; $a_0$, $b_0$ and $a_0^2 + b_0^2$ are strict in $\mathcal{M}_0^+$; and 
	$$U_0^* a U_0 = \begin{bmatrix} a_0^2 & a_0 b_0 \\ a_0 b_0 & p - a_0^2 \end{bmatrix}, \quad U_0^* b U_0 = \begin{bmatrix} b_0^2 & - a_0 b_0 \\ - a_0 b_0 & p - b_0^2 \end{bmatrix}.$$ 
	Now, $\begin{bmatrix} a_0^2 & a_0 b_0 \\ a_0 b_0 & p - a_0^2 \end{bmatrix} = \begin{bmatrix} 0 & 0 \\ 0 & p - a_0^2 - b_0^2 \end{bmatrix} + \begin{bmatrix} a_0^2 & a_0 b_0 \\ a_0 b_0 & b_0^2 \end{bmatrix}$. 
	
	Put $A_0 = \begin{bmatrix} a_0^2 & a_0 b_0 \\ a_0 b_0 & b_0^2 \end{bmatrix}$ and $W = \begin{bmatrix} b_0 & - a_0 \\ a_0 & b_0 \end{bmatrix}$. Then $W^* W = W W^* = (a_0^2 + b_0^2) \otimes I_2$ and $W^* P_0 W = A_0$. Consider the polar decomposition $W = U \vert W \vert$ for some partial isometry $U = \begin{bmatrix} u_1 & u_2 \\ u_3 & u_4 \end{bmatrix} \in M_2(\mathcal{M}_0)$. Then 
	$$b_0 = (a_0^2 + b_0^2)^{\frac 12} u_1 = (a_0^2 + b_0^2)^{\frac 12} u_4$$ 
	and 
	$$a_0 = - (a_0^2 + b_0^2)^{\frac 12} u_2 = (a_0^2 + b_0^2)^{\frac 12} u_3.$$ 
	Since $a_0^2 + b_0^2$ is strict in $\mathcal{M}_0^+$, we get $u_1 = u_4$ and $u_2 = - u_3$. Also, as $a_0^2 + b_0^2 = (a_0^2 + b_0^2) (u_1^2 + u_3^2)$, by the strictness of $a_0^2 + b_0^2$ in $\mathcal{M}_0^+$, we further conclude that $u_1^2 + u_3^2 = 1$. Now it is easy to verify that $U$ is a unitary. We show that $U$ is a strict unitary in $M_2(\mathcal{M}_0)$.
	
	Assume that $\mathcal{M}_0$ is acting on the Hilbert space $L$ an let $\xi \in L$. Since $b_0, a_0^2 + b_0^2 \in \mathcal{M}_0^+$, we have 
	$$0 \le \langle b_0 \xi, \xi \rangle = \langle (a_0^2 + b_0^2)^{\frac 12} u_1 \xi, \xi \rangle = \langle u_1 (a_0^2 + b_0^2)^{\frac 14} \xi, (a_0^2 + b_0^2)^{\frac 14} \xi \rangle.$$
	As $ r(a_0^2 + b_0^2) = 1$, $\lbrace (a_0^2 + b_0^2)^{\frac 14} \xi: \xi \in L \rbrace$ is dense in $L$ so $u_1 \in \mathcal{M}_0^+$. Similarly, we can conclude that $u_3 \in \mathcal{M}_0^+$.
	
	Let $p \in \mathcal{M}_0$ be a projection such that $p \le u_1$. Then $u_1 p = p$ and consequently,  
	$$b_0 p = (a_0^2 + b_0^2)^{\frac 12} u_1 p = (a_0^2 + b_0^2)^{\frac 12} p.$$ 
	Thus $b_0^2 p = (a_0^2 + b_0^2) p$ so that $a_0^2 p = 0$. Since $a_0$ is strict in $\mathcal{M}_0^+$, we get $p = 0$. Therefore, $s(u_1) = 0$. 
	
	 As $r(b_0) = 1 = r(a_0^2 + b_0^2)$ and $b_0 = (a_0^2 + b_0^2)^{\frac 12} u_1$, we can conclude that $r(u_1) = 1$. Thus $u_1$ is strict in $\mathcal{M}_0^+$. Since $u_1^2 + u_3^2 = 1$, $u_3$ is also strict in $\mathcal{M}_0^+$. Therefore, $U$ is a strict unitary. Put $P = U^* P_0 U$. Then $P$ is a strict projection and we have 
	$$A_0 = W^* P_0 W = \vert W \vert^2 P = \left( (a_0^2 + b_0^2) \otimes I_2 \right) P.$$

	Now putting $x_0 = a_0^2 + b_0^2$, (2) can now be verified in a routine way.
\end{proof} 
\begin{corollary}
	Let $\mathcal{M}$ be a von Neumann algebra with the underlying Hilbert space $H$. Then $a, b \in [0, 1]_{\mathcal{M}}$ is a strict and absolutely compatible pair, if and only if $$U^* a U = \left((p - x_0) \otimes I_2 \right) P + (x_0 \otimes I_2) P_0$$
	and 
	$$U^* b U = \left((p - x_0) \otimes I_2 \right) P + (x_0 \otimes I_2) P_1$$ 
	where $U$ is a unitary in $\mathcal{M}$, $x_0$ is a positive strict element of an abelian von Neumann subalgebra $\mathcal{M}_0$ with $p$ as its unity, $P_0 = \begin{bmatrix} 0 & 0 \\ 0 & p \end{bmatrix}$, $P_1 = \begin{bmatrix} p & 0 \\ 0 & 0 \end{bmatrix}$ and $P$ is a strict projection in $M_2(\mathcal{M}_0)$ of the form $P = \begin{bmatrix} a_0^2 & a_0 (p - a_0^2)^{\frac 12} \\ a_0 (p - a_0^2)^{\frac 12} & p - a_0^2 \end{bmatrix}$ for some strict element $a_0 \in \mathcal{M}_0^+$. 
\end{corollary}
\begin{proof}
	Let $\mathcal{M}_0$ be an abelian von Neumann algebra and let $P$ be any strict projection in $M_2(\mathcal{M}_0)$. Then by Lemma \ref{6}, there exists a strict unitary $U$ in $M_2(\mathcal{M}_0)$ such that $P = U^* P_0 U$. Thus $P' = U^* P_1 U$. Hence by Theorems \ref{9} and \ref{10}, we deduce that an absolutely compatible pair of strict elements $a$ and $b$ in a von Neumann algebra $\mathcal{M}$ is given by  
	$$U^* a U = \left((p - x_0) \otimes I_2 \right) P + (x_0 \otimes I_2) P_0$$
	and 
	$$U^* b U = \left((p - x_0) \otimes I_2 \right) P + (x_0 \otimes I_2) P_1$$ 
	where $U$ is a unitary in $\mathcal{M}$, $x_0$ is a positive strict element of an abelian von Neumann subalgebra $\mathcal{M}_0$ with $p$ as its unity, $P_0 = \begin{bmatrix} 0 & 0 \\ 0 & p \end{bmatrix}$, $P_1 = \begin{bmatrix} p & 0 \\ 0 & 0 \end{bmatrix}$ and $P$ is a strict projection in $M_2(\mathcal{M}_0)$. 
	
	If $Q$ be a strict projection in $M_2(\mathcal{M}_0)$, then by Corollary \ref{5}, we have  
	$$Q = \begin{bmatrix} b_0^2 & w b_0 (p - b_0^2)^{\frac 12} \\ w^* b_0 (p - b_0^2)^{\frac 12} & p - b_0^2 \end{bmatrix} = \begin{bmatrix} w & 0 \\ 0 & 1p \end{bmatrix} \begin{bmatrix} b_0^2 & b_0 (p - b_0^2)^{\frac 12} \\ b_0 (p - b_0^2)^{\frac 12} & p - b_0^2 \end{bmatrix} \begin{bmatrix} w & 0 \\ 0 & p 	\end{bmatrix}^*$$ 
	for some strict element $b_0 \in \mathcal{M}_0^+$ and a unitary $w \in \mathcal{M}_0$. As $w$ is a unitary in $\mathcal{M}_0$, $W_0 := \begin{bmatrix} w & 0 \\ 0 & p 	\end{bmatrix}$ is a unitary in $M_2(\mathcal{M}_0)$ such that $W_o^* P_0 W_0 = P_0$ and $W_o^* P_1 W_0 = P_1$. Hence in the above said description of $a$ and $b$, we may further assume that $P = \begin{bmatrix} a_0^2 & a_0 (p - a_0^2)^{\frac 12} \\ a_0 (p - a_0^2)^{\frac 12} & p - a_0^2 \end{bmatrix}$ for some strict element $a_0 \in \mathcal{M}_0^+$.
\end{proof}
\section{Absolute compatibility in $\mathbb{M}_2$ and Poincar\'{e} sphare}

In this section, we shall discuss a geometric aspect of absolute compatibility in $\mathbb{M}_2$. Let us note that strict elements in $\mathbb{C}$ are precisely those complex numbers $z$ such that $0 < \vert z \vert < 1$. Thus a strict unitary in $\mathbb{M}_2$ is the unitary matrix $U = \begin{bmatrix} z_1 & z_2 \\ z_3 & z_4 \end{bmatrix}$ where $0 < \vert z_i \vert < 1$. Further, a strict projection in $\mathbb{M}_2$ is a projection $P = \begin{bmatrix} a & \alpha  \\ \bar{\alpha} & 1 - a \end{bmatrix}$ where $0 < a < 1$ and $\vert \alpha \vert^2 = a (1 - a)$. In other words, strict projections are precisely rank one projections other than $P_0 = \begin{bmatrix} 0 & 0 \\ 0 & 1 \end{bmatrix}$ and $P_1 = P_0' = \begin{bmatrix} 1 & 0 \\ 0 & 0 \end{bmatrix}$. We shall call $P_0$ and $P_1$ \emph{pivotal projections}. (A reason for this nomenclature will be apparent later.) In the light of this description, we reformulate the results of Section 3 in the the context of $\mathbb{M}_2$. 
\begin{theorem}\label{11}
	Let $A, B \in [0, I_2]$ where $I_2$ is the identity matrix in $\mathbb{M}_2$. Then $A$ and $B$ are strict and absolutely compatible in $\mathbb{M}_2^+$ if and only if there exist rank one projections $P$ and $Q$ in $\mathbb{M}_2$ with $P \notin \lbrace Q, Q' \rbrace$ and a rea number $\lambda$ with $0 < \lambda < 1$ such that $A = (1 - \lambda) P + \lambda Q$ and $B = (1 - \lambda) P + \lambda Q'$.
\end{theorem}
\begin{proof}
	First we assume that $A$ and $B$ are strict and absolutely compatible in $\mathbb{M}_2^+$. Then by Theorem \ref{10}, there exist a unitary $U \in \mathbb{M}_2$, a rank one projection $R \notin \lbrace P_0, P_1 \rbrace$ and a real number $0 < \lambda < 1$ such that $U^* A U = (1 - \lambda) P_0 + \lambda R$ and $U^* B U = (1 - \lambda) P_0 + \lambda R'$. Put $P = U P_0 U^*$ and $Q = U R U^*$. Then $P$ and $Q$ are rank one projections in $\mathbb{M}_2$ with with $P \notin \lbrace Q, Q' \rbrace$ and a real number $\lambda$ with $0 < \lambda < 1$ such that $A = (1 - \lambda) P + \lambda Q$ and $B = (1 - \lambda) P + \lambda Q'$.
	
	Conversely, we assume that there exist rank one projections $P$ and $Q$ in $\mathbb{M}_2$ with $P \notin \lbrace Q, Q' \rbrace$ and a real number $\lambda$ with $0 < \lambda < 1$ such that $A = (1 - \lambda) P + \lambda Q$ and $B = (1 - \lambda) P + \lambda Q'$. Find a unitary $W \in \mathbb{M}_2$ such that $P = W^* P_0 W$. Put $R = W Q W^*$. Then $R$ is a strict projection in $\mathbb{M}_2$ as $R \notin \lbrace P_0, P_1 \rbrace$. Also we have $W A W^* = (1 - \lambda) P_0 + \lambda R$ and $W B W^* = (1 - \lambda) P_0 + \lambda R'$. Thus by Theorem \ref{9}, we have $A$ and $B$ are strict and absolutely compatible in $\mathbb{M}_2^+$. 
\end{proof}
This simplified characterization of absolute compatibility of a strict pair of elements in $\mathbb{M}_2$ leads us to a geometric interpretation. Let us recall that the set of rank one projections in $\mathbb{M}_2$ can be identified with Poincar\'{e} sphere (or Bloch sphere). Among the several identifications, we choose the real affine isomorphism 
$$\begin{bmatrix} a & \alpha \\ \bar{\alpha} & 1 - a \end{bmatrix} \mapsto \left(a, \mathfrak{Re}\alpha, \mathfrak{Im}\alpha \right)$$  
between the set $\mathcal{P}_2$ of rank one projections in $\mathbb{M}_2$ and sphere $$\mathcal{B}_d = \left\lbrace (x, y, z): \left(x - \frac 12\right)^2 + y^2 + z^2 = \left(\frac 12\right)^2 \right\rbrace$$ 
in $\mathbb{R}^3$. It follows from \cite[Remark 3.4]{JKP} that if $A$ and $B$ are strict and absolutely compatible in $\mathbb{M}_2$, then $A$ and $B$ belong to the set 
$$\mathcal{S} = \lbrace X \in \mathbb{M}_2^+: 0 < \det(X) < \frac 14 ~ \textrm{and} ~ trace(X) = 1 \rbrace.$$ 
The above said identification can be extended to $\mathcal{S}$. Under this map, $\mathcal{S}$ corresponds to the open ball 
$$\mathcal{B}^o = \left\lbrace (x, y, z): \left(x - \frac 12\right)^2 + y^2 + z^2 < \left(\frac 12\right)^2 \right\rbrace$$ 
punctured at the centre $\left( \frac 12, 0, 0 \right)$. (The centre corresponds to $\frac 12 I_2$ in $\mathbb{M}_2$ which is strict but is absolutely compatible precisely with projections in $\mathbb{M}_2$ only. That is, it does not find any absolutely compatible couple in $\mathcal{S}$.) 

It was noted in \cite{JKP} that given $A \in \mathcal{S}$, the set of absolutely couples of $A$ describes an eliptic path in $\mathcal{S}$. More precisely, let $A, X \in \mathcal{S}$ be such that $A$ is absolutely compatible with $X$. If we identify $A$ and $X$ with points in $\mathcal{B}_0$ and denote them as $A$ and $X$ again, then the locus of $X$ is a prolate spheroid with $A$ and $A' := I_2 - A$ at the foci and the major axis as the diameter joining $A$ and $A'$ \cite[Theorem 3.7]{JKP}. Theorem \ref{11} leads us to another geometric aspect of absolute compatibility in $\mathbb{M}_2$ which we describe below. 
\begin{enumerate}
	\item Let $A, B \in \mathcal{S}$ be such that $A$ is absolutely compatible with $B$. By Theorem \ref{11}, there exist $P, Q \in \mathcal{P}_2$ with $P \notin \lbrace Q, Q' \rbrace$ and a real number $\lambda$ with $0 < \lambda < 1$ such that $A = (1 - \lambda) P + \lambda Q$ and $B = (1 - \lambda) P + \lambda Q'$. Let the corresponding points in $\mathcal{B} := \mathcal{B}^0 \bigcup \mathcal{B}_d$ be denoted by the same symbols. Then the sphere, described by $A$ and $B$ as extremities of a diameter, touches $\mathcal{B}_d$ at $P$ in such a way that the points $P$, $P'$, $Q$, $Q'$, $A$ and $B$ are coplanar and the line joining $A$ and $B$ is parallel to the line joining $Q$ and $Q'$. 
	\item Conversely, assume that $P, Q \in \mathcal{B}_d$ with $P \notin \lbrace Q, Q' \rbrace$. As $Q Q'$ is a diameter of the sphere $\mathcal{B}_d$, it subtends a right angle at $P$. Let $\lambda$ be a real number with $0 < \lambda < 1$ and consider the points $A = (1 - \lambda) P + \lambda Q$ and $B = (1 - \lambda) P + \lambda Q'$. Then the line-segment $A B$ also subtends a right angle at $P$. Thus the sphere passing through $A$, $B$ and $P$ touches the sphere $\mathcal{B}_d$ at $P$ such that $A B$ is a diameter. 
	\item The sphere described by a strict pair of absolutely compatible elements in $\mathbb{M}_2$ can be determined by two parameters $P \in \mathcal{B}_d$ and a real number $\lambda$ with $0 < \lambda < 1$. This sphere is called a \emph{pivotal Poincar\'{e} sphere} of index $\lambda$ with \emph{pivot} at $P$. This is denoted by $P_\lambda$. Note that the standard Poinar\'{e} sphere is pivotal at every point on the sphere and its index is $1$.
	\item A pivotal Poincar\'{e} sphere of index $\lambda$ pivoted at $P$ of the Poincar\'{e} sphere can be defined independent of the points $A$ and $B$. Consider the point $M = (1 - \lambda) P + \lambda P'$.
	Then $P_\lambda$ is described as the sphere having $P M$ as a diameter. A high school geometric construction can used to show that $A B$ is another diameter of this sphere. In fact, the chords $P Q$ and $P Q'$ meet the sphere $P_\lambda$ at the points $A$ and $B$ respectively. (Figure.1.)
	
	{\centering
		
		\begin{tikzpicture}
			\draw (2,2) circle (3cm);
			\draw (0, 2) circle (1cm);
			\rotatebox[origin=c]{-73}{\draw[dashed] (-1.31,2.6) ellipse (3cm and .9cm);}
			\rotatebox[origin=c]{-74}{\draw[dashed] (-1.9,.6) ellipse (1cm and .3cm);}
			\draw (-1,2) -- (5,2);
			\draw (-1,2) -- (1.2,4.9) -- (3,-.8) -- (-1,2);
			\draw (-.25,2.95) -- (.35,1.05);
			\coordinate (P) at (-1,2);
			\node[left] at (P) {$P$};
			\coordinate (P') at (5,2);
			\node[right] at (P') {$P'$};
			\coordinate (Q) at (1.2,4.9);
			\node[above] at (Q) {$Q$};
			\coordinate (Q') at (3,-.8);
			\node[below] at (Q') {$Q'$};
			\coordinate (A) at (-.25,2.95);
			\node[above left] at (A) {$A$};
			\coordinate (B) at (.35,1.05);
			\node[below] at (B) {$B$}; 
			\coordinate (Q') at (3,-1.5);
			\node[below left] at (3,-1.5) {Figure.1};
		\end{tikzpicture} 
	}	
	\item Let $C$ be any point on the sphere $P_{\lambda}$ and let $D$ be the point diametrically opposite to $C$ on $P_{\lambda}$. Then there exists a unique point $R$ on $\mathcal{B}_d$ such that $C = (1 - \lambda) P + \lambda R$. Also then $D = (1 - \lambda) P + \lambda R'$ where $R'$ is the point diametrically opposite to $R$ on $\mathcal{B}_d$. Conversely, given a point $R$ in $\mathcal{B}_d$ there exists a unique pair of points $C$ and $D$ on $P_{\lambda}$ such that $C = (1 - \lambda) P + \lambda R$ and $D = (1 - \lambda) P + \lambda R'$ where $R'$ is diametrically opposite to $R$ in $\mathcal{B}_d$. 
	\item If we denote these points as the corresponding matrices in $\mathbb{M}_2$, we conclude that given a projection $P$ and a real number $\lambda$, $0 < \lambda < 1$, there exist a bijection between a pair of orthogonal projections $R$ and $R'$ (other than $P$ and $P'$) and a pair of strict, compatible pair $C$ and $D$ (lying on the pivotal poincar\'{e} sphere $P_{\lambda}$) such that $C = (1 - \lambda) P + \lambda R$ and $D = (1 - \lambda) P + \lambda R'$. (Note that in this case, the projection $P$ is the pivot.) 
	\item We expect that this geometric realization of a pair of strict and absolutely compatible pair of (positive) matrices in $\mathbb{M}_2$ may throw more light on the study of polarization of a beam of light in optics. We note that $\mathcal{P}_2$ describes the set of all pure states and the set $\mathcal{S}$ describes the set of other (mixed) states of (the von Neumann algebra) $\mathbb{M}_2$.
	
	In optics, it is observed that if the polarization states of two beams of light are orthogonal to each other, then they have the same ellipticity with orientations opposite to each other. This observation describes the nature of polarization in a free system and the corresponding states are pure, that is the points of Poincar\'{e} sphere.  We understand that this study can be extended to describe a relation between geometric properties of polarization of a pair of beams of light in a system having obstacles if their corresponding states are absolutely compatible to each other. The author is in an advanced stage to perform a suitable experiment in collaboration with his colleagues working in Optics. 
\end{enumerate}

\thanks{{\bf Acknowledgements}: The author is grateful to Antonio M. Peralta for expressing his belief that absolute compatibility must be closely related to projections. This research was partially supported by Science and Engineering Research Board, Department of Science and Technology, Government of India sponsored  Mathematical Research Impact Centric Support project (reference no. MTR/2020/000017).}

\end{document}